\documentclass[11pt,twoside]{article}
\usepackage{amsmath}
\usepackage{latexsym}
\usepackage{amssymb,amsbsy,amsmath,amsfonts,amssymb,amscd,mathtools}
\usepackage{graphicx,float}

  \usepackage{graphics} 
  \usepackage{epsfig} 
 \usepackage[colorlinks=true]{hyperref}
 \usepackage{mathtools}
\hypersetup{urlcolor=blue, citecolor=red}
\usepackage[smallfootnotes=false]{hypdvips}
\newtheorem{theorem}{Theorem}[section]

\newtheorem{proof}{Proof}
\newtheorem{lemma}[theorem]{Lemma}

\newtheorem{remark}{Remark}
\newcommand{\ep}{\varepsilon}

\newtheorem{example}[theorem]{Example}

\title{\Large \bf Adaptation and migration of a population between patches}

\author{Sepideh Mirrahimi  \thanks{
CMAP, Ecole Polytechnique, CNRS, INRIA. Route de Saclay, 91128 Palaiseau Cedex, France.  Email: mirrahimi@cmap.polytechnique.fr.}
}

\date{\today}

\begin{document}
\maketitle

\pagestyle{plain}
\pagenumbering{arabic}

\begin{abstract}
 A Hamilton-Jacobi formulation has been established previously for phenotypically structured population models where the solution concentrates as Dirac masses in the limit of small diffusion. Is it possible to extend this approach to spatial models? Are the limiting solutions still in the form of sums of Dirac masses? Does the presence of several habitats lead to polymorphic situations? 

We study the stationary solutions of a structured population model, while the population is structured by continuous phenotypical traits and discrete positions in space. The growth term varies from one habitable zone  to another, for instance because of a change in the  temperature. The individuals can migrate from one zone to another with a constant rate. The mathematical modeling of this problem, considering mutations between phenotypical traits and competitive interaction of individuals within each zone via a single resource, leads to a system of coupled parabolic integro-differential equations. We study the asymptotic behavior of the stationary solutions to this model in the limit of small mutations. The limit, which is a sum of Dirac masses, can be described with the help of  an effective Hamiltonian. The presence of migration can modify the dominant traits and lead to polymorphic situations. 
\end{abstract}

\noindent{\bf Key-Words:} Structured populations, phenotypical and spatial structure, Hamilton-Jacobi equation, viscosity solutions, Dirac concentrations, stationary solutions

\noindent{\bf AMS Class. No:} 35B25,  47G20, 49L25, 92D15

\section{Introduction}
\label{sec:intro}

Non-local Lotka-Volterra equations arise in models of adaptive evolution of phenotypically structured populations. These equations have the  property that the solutions concentrate generally, in the limit of small diffusion, on several isolated points, corresponding to distinct traits. Can we generalize these models by adding a spatial structure? How do the dominant traits evolve if we introduce a new habitat? To understand the interaction of ecological and evolutionary processes in population dynamics, spatial structure of the communities and adaptation of species to the environmental changes, it is crucial to dispose mathematical models that describe them jointly. We refer to \cite{SL.MB:08} and the references therein for general literature on the subject.  In this manuscript we consider a model where several distinct favorable habitable zones are possible. Population dynamics models  structured by spatial patches have been studied using both deterministic and probabilistic methods (see for instance \cite{AS.GM:09,VB.AL:12}). Our model, in the case of two patches, is indeed  very close to the one studied in \cite{AS.GM:09} where the authors use adaptive dynamics theory (adaptive dynamics is a theory, based on dynamical systems and their stability, to study population dynamics \cite{OD:04}). Here we model similar phenomena, by adding a spatial structure to an earlier known integro-differential model describing the darwinian evolution. Integro-differential models have the advantage that the mutations can be considered directly in the model without assuming a separation of time scales of evolution and ecology.  The present work provides a general description of the asymptotic stationary solutions, in the general case where two or several patches are possible.

We study the asymptotic behavior of solutions of a system of coupled elliptic integro-differential equations with small diffusion terms. These solutions are the stationary solutions to a parabolic system describing the dynamics of a population density.  The individuals are characterized by phenotypical traits, that we denote by $x\in \mathbb{R}^d$. They can move between two or several patches, which are favorable habitable zones, with constant rates (that we denote by $\nu^1$ and $\nu^2$ in the case of two patches). The mathematical modeling is based on the darwinian evolution and takes into account mutations and competition between the traits. There is a large literature for mathematical modeling and analysis on the subject of adaptive evolution, we refer the interested reader to \cite{SG.EK.GM.JM:97,SG.EK.GM.JM:98,OD:04,OD.PJ.SM.BP:05,GM.MG:05,LD.PJ.SM.GR:08}. Here, we represent the birth and death term by a net growth term $R^i(x,I^i)$ that is different in each patch, for instance because of a change in the temperature, and depends on the integral parameter $I^i$, which corresponds to the the pressure exerted by the whole population within patch $i$ on the resource. To model the mutations, we use Laplace terms with a small rate $\ep$ that is introduced to consider only rare mutations. We study the asymptotic behavior of stationary solutions as the mutation rate $\ep$ goes to $0$. The asymptotic solutions are generally concentrated on one or several Dirac masses. We describe the position and the weight of these Dirac masses using a Hamilton-Jacobi approach.

The time-dependent model,  in the case of two patches, is written as
\begin{equation}\label{main}\left\{\begin{array}{rlr}
\partial_t n_\ep^1-\ep \Delta n_\ep^1&=\frac 1\ep n_\ep^1 R^1(x,I_\ep^1)+\frac 1 \ep \nu^2 n_\ep^2-\frac 1 \ep\nu^1 n_\ep^1,&
\\&& \quad x\in \mathbb{R}^d,\\
\partial_t n_\ep^2-\ep \Delta n_\ep^2&=\frac 1 \ep n_\ep^2 R^2(x,I_\ep^2)+\frac 1 \ep \nu^1 n_\ep^1-\frac 1\ep\nu^2 n_\ep^2,&
\end{array}\right.
\end{equation}
with 
\begin{equation}\label{Iei}
I_\ep^i=\int \psi^i(x)n_\ep^i(x)dx,\qquad \text{for $i=1,2.$}
\end{equation}
Such models, without the structure in space, have been derived from stochastic individual based models in the limit of large populations (see \cite{NC.RF.SM:08,NC..RF.SM:08bis}). This  manuscript follows earlier works on parabolic Lotka-Volterra type equations to study concentration effects in models of phenotypically structured populations, that are based on a Hamilton-Jacobi formulation (see \cite{OD.PJ.SM.BP:05,GB.BP:08,GB.SM.BP:09,AL.SM.BP:10}). The novelty of our work is that we add a spatial structure to the model by considering a finite number of favorable habitable zones. We thus have a system instead of a single equation. A Hamilton-Jacobi approach in the case of systems has also been introduced in \cite{JC.SC.BP:07} for an age structured model. See also \cite{AC.SC:05} for a study of stationary solutions of the latter system. The Hamilton-Jacobi approach can also be used in problems other than adaptive evolution to prove concentration phenomena. See for instance \cite{BP.PS:09bis,BP.PS:09,SM.PS:12} where  related methods have been used to study the motion of motor proteins. 

We are interested in the equilibria of \ref{main} limited to a bounded domain, that are given by solutions of the following system
\begin{equation}\label{stat}\left\{\begin{array}{rll}
-\ep^2 \Delta n_\ep^1&=n_\ep^1 R^1(x,I_\ep^1)+\nu^2 n_\ep^2-\nu^1 n_\ep^1&\quad \text{in $B_L(0)$},\\\\
-\ep^2 \Delta n_\ep^2&=n_\ep^2 R^2(x,I_\ep^2)+\nu^1 n_\ep^1-\nu^2 n_\ep^2&\quad \text{in $B_L(0)$},\\\\
\nabla n_\ep^i\cdot \vec{n}&=0&\quad \text{in $\partial B_L(0)$ and for $i=1,\,2$},
\end{array}\right.
\end{equation}
where $B_L(p)$ is a ball of radius $L$ with center in $p$ and $\vec{n}(x)$ is the unit normal vector, at the point $x\in \partial B_L(0)$, to the boundary of $B_L(0)$. The Neumann boundary condition is a way to express that mutants cannot be born in $\mathbb{R}^d\setminus B_L(0)$.

To formulate our results we introduce the assumptions we will be using throughout the paper. We assume that, there exist positive constants $a_m$ and $a_M$ such that
\begin{equation}\label{as:psi}\psi^1=\psi^2=\psi,\quad a_m\leq\psi(x)\leq a_M, \quad  \|\psi(x)\|_{W^{2,\infty}}\leq A \quad \text{and}\quad \nabla \psi\cdot  \vec{n}=0 \text{ in $\partial B_L(0)$}.\end{equation}
Moreover there exist positive constants $I_m$, $I_M$, $\delta$ and $C$ such that, for all $x\in B_L(0)$ and $i,j=1,\,2$,
\begin{equation}\label{as:quad1}
\delta\leq\min \left(R^i(x,\frac{\nu^j}{\nu^i} I_m),R^i(x,I_m)\right),\quad\max \left(R^i(x,\frac{\nu ^j }{\nu^i} I_M),R^i(x,I_M)\right)\leq -\delta,
\end{equation}
\begin{equation}\label{as:quad2}
 -C\leq \frac{\partial R^i}{\partial I}(x,I)\leq -\frac 1 C,
\end{equation}
\begin{equation}\label{as:R2}
-D\leq R_{\xi\xi}^i(x,I),\quad \text{for $x\in B_L(0)$, $I\in [I_m,I_M]$, $\xi\in \mathbb{R}^d$, $|\xi|=1$ and $i=1,2$.}
\end{equation}
We use the Hopf-Cole transformation
\begin{equation} \label{Hopf}n_\ep^i =\exp(\frac{u_\ep^i}{\ep})  ,\qquad \text{for $i=1,2,$}\end{equation}
and replace the latter in the system satisfied by $n_\ep^i$ to obtain
\begin{equation}\label{equ}
\left\{\begin{array}{rll}
-\ep \Delta u_\ep^1&=|\nabla u_\ep^1|^2+ R^1(x,I_\ep^1)+\nu^2 \exp(\frac{u_\ep^2-u_\ep^1}{\ep})-\nu^1,&\quad \text{in $B_L(0)$},\\\\
-\ep \Delta u_\ep^2&=|\nabla u_\ep^2|^2+R^2(x,I_\ep^2)+\nu^1 \exp(\frac{u_\ep^1-u_\ep^2}{\ep}) -\nu^2&\quad \text{in $B_L(0)$},\\\\
 \nabla u_\ep^i\cdot \vec{n}&=0&\quad \text{in $\partial B_L(0)$} \\
 &&\text{and for $i=1,\,2$}.
\end{array}\right.
\end{equation}
We prove the following
\begin{theorem}\label{thm:main}
Assume \ref{as:psi}--\ref{as:quad2}. Then, as $\ep\to 0$ along subsequences, both sequences $(u_\ep^1)_\ep$ and $(u_\ep^2)_\ep$ converge uniformly in $B_L(0)$ to a continuous function $u\in \mathrm{C}(B_L(0))$ and $(I_\ep^1,I_\ep^2)$ goes to $(I^1,I^2)$, with $(u,I^1,I^2)$ such that $u$ is a viscosity solution to the following equation
\begin{equation}\label{HJ}\left\{ \begin{array}{ll}-|\nabla u|^2= H(x,I^1,I^2),&\quad \text{in $B_L(0)$},\\\\
\max_{x\in B_L(0)}u(x)=0,\end{array}\right.\end{equation}
with 
\begin{equation}\begin{array}{c} \text{$H(x,I^1,I^2)$ the largest eigenvalue of the matrix}\\\displaystyle \label{def:H}\mathcal{A}=\left(\begin{array}{cc} R^1(x,I^1)-\nu^1&\nu^2\\
\nu^1 & R^2(x,I^2)-\nu^2\end{array}\right).\end{array}\end{equation}
\end{theorem}

The function $H$ is indeed an effective Hamiltonian that contains information from the two patches and helps us in Theorem \ref{thm:as} to describe the support of  the weak limits of $(n_\ep^1,n_\ep^2)$ as $\ep\to 0$. 
We can interpret $H(x,I^1,I^2)$ as the fitness
of the system in the limit of $\ep\to 0$  (see \cite{JM.RN.SG:92} for the definition of fitness). 

The difficulty here is to find appropriate regularity estimates on $u_\ep^i$, that we obtain using the Harnack inequality \cite{JB.MS:04} and the Bernstein method \cite{C.I.L:92}. To prove convergence to the Hamilton-Jacobi equation, we are inspired from the method of perturbed test functions in homogenization \cite{LE:89}.

The above information on the limit of $u_\ep^i$ allows us to describe the limit of the densities $n_\ep^i$ as $\ep$ vanishes. We prove

\begin{theorem}\label{thm:as}
Assume \ref{as:psi}--\ref{as:R2}. Consider a subsequence such that $u_\ep^1$ and $u_\ep^2$ converge uniformly to $u\in C\left(B_L(0)\right)$ and $(I_\ep^1,I_\ep^2)$ goes to $(I^1,I^2)$, as $\ep \to 0$, with $(u,I^1,I^2)$ solution of \ref{HJ}. Let $n_\ep^i$, for $i=1,2$, converge weakly in the sense of measures to $n^i$ along this subsequence.  We have
 \begin{equation} \label{eq:supp}\mathrm{supp } \;n^i\subset \Omega\cap \Gamma,\quad \text{for $i=1,2$,}\end{equation}
  with
  \begin{equation}\label{Gamma}
  \begin{array}{c}
\Omega=\{x\in B_L(0) \,| \, u(x)=0\},\\ \Gamma=\{x\in  B_L(0)\,| \, H(x,I^1,I^2)=\max_{x\in B_L(0)}H(x,I^1,I^2)=0\}.
\end{array}
\end{equation}
Moreover, we have
\begin{equation}\label{dist}
\left(R^1(x)-\nu^1\right)n^1(x)+\nu^2n^2(x)=0,\quad \left(R^2(x)-\nu^2\right)n^2(x)+\nu^1n^1(x)=0, \quad \text{in }B_L(0)
\end{equation}
in the sense of distributions. The above condition is coupled by
\begin{equation}\label{eq:I1I2first}
\int_{B_L(0)} \psi^i(x)n^i(x)=I^i.
\end{equation}
\end{theorem}

Theorem \ref{thm:as} provides us  with a set of algebraic constraints on the limit, which allows us to describe the latter. 
In particular, if the support of $n^i$, for $i=1,2$, is a set of distinct points: $\mathrm{supp}\,n^i\subset \{x_1,x_2,\cdots,x_k\}$,  \ref{dist} implies that
\begin{equation}\label{eq:ni}
n^i=\sum_{j=1}^k\rho^i_j \delta(x-x_j),\qquad \text{for $i=1,2$,}\end{equation}
with
\begin{equation}\label{eq:rho}
\rho_j^2=\rho_j^1\left(\frac{\nu^1-R^1(x_j,I^1)}{\nu^2}\right)=\rho_j^1\left(\frac{\nu^1}{\nu^2-R^2(x_j,I^2)}\right).\end{equation}
Furthermore, the weights  $(\rho^i_1,\cdots,\rho^i_k)$ satisfy the normalization condition
\begin{equation}\label{eq:I1I2}
\sum_{j=1}^k \psi^i(x_j)\rho^i_j=I^i,\qquad \ \text{for $i=1,2$.}\end{equation}
Condition \ref{eq:rho} means that the vector $\left(\begin{array}{c} \rho_j^1\\\rho_j^2 \end{array}\right)$ is the eigenvector corresponding to the largest eigenvalue of the matrix $\mathcal{A}$ at the point $x_j$, which is $0$. Thereby \ref{eq:rho} implies once again that $\mathrm{supp}\,n^i\subset \Gamma$. 


We point out that since $n^i$, for $i=1,2$, is such that the fitness $H$ vanishes on the support of $n^i$ and is negative outside the support, we can interpret $n^i$ as evolutionary stable distribution of the model. In adaptive dynamics, evolutionary stable distribution (ESD) corresponds to a distribution that remains stable after introduction of  small mutants (see \cite{JM.GP:73,IE:83,PJ.GR:09} for a more detailed definition).  See also \cite{LD.PJ.SM.GR:08,Raoul2009-3} for related works on  stability and convergence to ESD for trait-structured integro-differential models. \\

The set of assumptions in Theorem \ref{thm:as} allows us to describe the asymptotics of the stationary solutions, in the limit of rare or small mutations. In Section \ref{sec:ex} we provide some examples where based on this information we can describe  the asymptotics. In particular, we notice that the introduction of a new environment can lead to dimorphic situations. We refer to \cite{NC.PJ:10} for a related work using the Hamilton-Jacobi approach, where polymorphic situations  can also appear in a model with multiple resources.\\

The paper is organized as follows. In Section \ref{sec:reg} we prove some bounds on $I_\ep$ and some regularity properties on $u_\ep$ that allow us to pass to the limit as $\ep\to 0$ and derive the Hamilton-Jacobi equation with constraint. Theorem \ref{thm:main} is proved in Section \ref{sec:HJ}. Using the results obtained on the asymptotic behavior of $(u_\ep^i)_\ep$ we prove Theorem \ref{thm:as} in Section \ref{sec:as}. In Section \ref{sec:ex} we provide some examples where the information given by Theorem \ref{thm:main} and Theorem \ref{thm:as} allows us to describe the limit. The asymptotic behavior of the stationary solutions in a more general framework, where more than two habitable zones are considered, is given in Section \ref{sec:ext}. Finally in Section \ref{sec:num} we present some numerical simulations for the time-dependent problem and compare them with the behavior of stationary solutions.

\section{Regularity results}\label{sec:reg}

\begin{lemma}Under assumptions \ref{as:psi}--\ref{as:quad2} we have, for $\ep\leq \ep_0$ chosen small enough,
\begin{equation}\label{boundI}
I_m\leq I_\ep^i\leq I_M, \qquad \text{for $i=1,\,2$}.
\end{equation}
In particular, along subsequences, $(I_\ep^1,I_\ep^2)_\ep$ converges to $(I^1,I^2)$, with $I_m\leq I^1,\; I^2\leq I_M$. 
\end{lemma}

\begin{remark}
This is the only part, where we use Assumption \ref{as:psi}. If $(n_\ep^1,n_\ep^2)$ is a solution of \ref{stat} such that \ref{boundI} is satisfied, then the results of Theorems \ref{thm:main} and \ref{thm:as} hold true without necessarily assuming  \ref{as:psi}. In particular, one can take $\psi^1\not\equiv \psi^2$.
\end{remark}

\begin{proof}
We prove the result by contradiction. We suppose that $I_\ep^1> I_M$ (the case with $I_\ep^2>I_M$, and the inequalities from below can be treated following similar arguments).
We multiply the first equation in \ref{stat}  by $\psi(x)$, integrate, and use \ref{as:psi} to obtain
$$-\ep^2\frac{A}{a_m}I_\ep^1\leq\int \psi(x) n_\ep^1(x)R^1(x,I_\ep^1)dx+\nu^2 I_\ep^2-\nu^1I_\ep^1.$$
Using now \ref{as:quad1}, \ref{as:quad2} and the fact that $I_\ep^1>I_M$ we deduce that, for $\ep\leq \ep_0$ small enough,
$$0\leq\left(\delta-\ep^2\frac{A}{a_m}\right)I_\ep^1 \leq \nu^2 I_\ep^2-\nu^1I_\ep^1,
$$
and thus 
$$\frac{\nu^1}{\nu^2}\,I_M\leq I_\ep^2.$$
Now we multiply the equations in \ref{stat}  by $\psi(x)$, integrate and add them and use \ref{as:psi} to obtain
$$-\ep^2\frac{A}{a_m}(I_\ep^1+I_\ep^2) \leq \int \psi(x) n_\ep^1(x)R^1(x,I_\ep^1)dx+\int \psi(x) n_\ep^2(x)R^2(x,I_\ep^2)dx.$$
From  \ref{as:quad1}, \ref{as:quad2} and  the above bounds on $I_\ep^1$ and $I_\ep^2$ it follows that
$$-\ep^2\frac{A}{a_m}(I_\ep^1+I_\ep^2) \leq -\delta(I_\ep^1+I_\ep^2),$$
which is not possible if $\ep$ is small enough. We conclude that $I_\ep^1\leq I_M$.

\end{proof}

\begin{theorem}\label{thm:reg}
Assume \ref{as:psi}--\ref{as:quad2}. Then \\
(i) there exists a positive constant $D$, such that for $\ep\leq \ep_0$,
\begin{equation} \label{harnack}
|u_\ep^i(x)-u_\ep^j(y)|\leq D\ep,\qquad \text{for all $x,\,y\in B_L(0)$, $|x-y|\leq \ep$ and $i,\,j\in \{1,2\}$.}
\end{equation}
(ii) For $i=1,\,2$ and all $\ep\leq \ep_0$, the family $(u_\ep)_\ep$ is uniformly Lipschitz and uniformly bounded from below.\\
(iii)  For all $a>0$, there exists $\ep_1=\ep_1(a)$ such that for all $\ep\leq \ep_1$,
\begin{equation}\label{lbound}u_\ep^i(x)\leq a,\qquad \text{for $x\in B_L(0)$ and $i=1,\,2$.}
\end{equation}

\end{theorem}
\begin{proof}
(i) We define 
$$\widetilde n_\ep^i(y)=n_\ep^i(\ep y),\qquad\text{for $i=1,\,2$.}$$
From \ref{stat} we have
\begin{equation}\label{tilden}\displaystyle\left\{\begin{array}{rll}
- \Delta\widetilde n_\ep^1&=\widetilde n_\ep^1 R^1(\ep x,I_\ep^1)+\nu^2\widetilde n_\ep^2-
\nu^1 \widetilde n_\ep^1&\quad \text{in $B_{\frac L \ep}(0)$},\\\\
- \Delta\widetilde  n_\ep^2&=\widetilde n_\ep^2 R^2(\ep x,I_\ep^2)+\nu^1\widetilde  n_\ep^1-\nu^2 \widetilde n_\ep^2&\quad \text{in $B_{\frac L\ep}(0)$},
\end{array}\right.
\end{equation}
Moreover, from \ref{as:quad1}, \ref{as:quad2} and \ref{boundI} we have, for $\ep\leq \ep_0$,
$$\delta-C(I_M-I_m)\leq R(\ep x,I_\ep)\leq -\delta+C(I_M-I_m).$$
Therefore the coefficients of the linear elliptic system \ref{tilden} are bounded uniformly in $\ep$. It follows from the classical Harnack inequality (\cite{JB.MS:04}, Theorem 8.2) that there exists a constant $D=D(C,I_m,I_M,\delta,\nu^1,\nu^2)$ such that for all $y_0\in B_{\frac L\ep}(0)$ such that $B_1(y_0)\subset B_{\frac L \ep}(0)$ and for $i,j=1,2$,
$$
\sup_{z\in B_1(y_0)}\widetilde n_\ep^i(z)\leq D\,\inf_{z\in B_1(y_0)}\widetilde n_\ep^j(z).$$
Rewriting the latter in terms of $ n_\ep^1$ and $ n_\ep^2$ and replacing $(y_0,z)$ by $(\frac{x}{\ep},\frac{z'}{\ep})$ we obtain
$$
\sup_{z'\in B_\ep(x)}n_\ep^i(z')\leq D\,\inf_{z'\in B_\ep(y_0)}n_\ep^j(z'),$$
and  thus from \ref{Hopf} we deduce \ref{harnack}.\\

(ii) To prove the Lipschitz bounds, we use the Bernstein method (see \cite{C.I.L:92}). We assume that
\begin{equation} \label{max}\max_{x\in B_L(0)}(|\nabla u_\ep^1(x)|,|\nabla u_\ep^2(x)|)=|\nabla u_\ep^1(x_\ep)|,\end{equation}
that is the maximum is achieved at a point $x_\ep\in B_L(0)$ and for $i=1$ (the case where the maximum is achieved for $i=2$ can be treated by similar arguments). From the Neumann boundary condition in \ref{equ} we know that $x_\ep$ is an interior point of $B_L(0)$. We define $p=|\nabla u_\ep^1|^2$ and notice that
$$\Delta p=2 \mathrm{Tr}\;(\mathrm{Hess}\;u_\ep^1)^2+2\nabla (\Delta u_\ep^1)\cdot \nabla u_\ep^1.$$
We now differentiate the first equation in \ref{equ} with respect to $x$ and multiply it by $\nabla u_\ep^1$ to obtain
$$-\ep\nabla(\Delta u_\ep^1)\cdot \nabla u_\ep^1=\nabla p\cdot \nabla u_\ep^1+\nabla R^1\cdot \nabla u_\ep^1+\nu^2\left(\frac{\nabla u_\ep^2-\nabla u_\ep^1}{\ep}\right)\cdot\nabla u_\ep^1 \exp(\frac{u_\ep^2-u_\ep^1}{\ep}).$$
From \ref{max} we have
$$\left(\nabla u_\ep^2(x_\ep)-\nabla u_\ep^1(x_\ep)\right)\cdot\nabla u_\ep^1(x_\ep)\leq 0,$$
and thus
$$-\frac\ep2\Delta p(x_\ep)+\ep \mathrm{Tr}\;(\mathrm{Hess}\;u_\ep^1(x_\ep))^2\leq \nabla p(x_\ep)\cdot \nabla u_\ep^1(x_\ep)+\nabla R^1(x_\ep)\cdot \nabla u_\ep^1(x_\ep).$$
Moreover from \ref{max} we have $\nabla p(x_\ep)=0$ and $\Delta p\leq 0$. It follows that
$$\ep \left(\Delta u_\ep^1(x_\ep)\right)^2\leq \ep d\, \mathrm{Tr}\;(\mathrm{Hess}\;u_\ep^1(x_\ep))^2\leq d\nabla R^1(x_\ep)\cdot \nabla u_\ep^1(x_\ep).$$
Using again \ref{equ} we obtain
$$\left( |\nabla u_\ep^1|^2+R^1(x_\ep,I_\ep^1)+\nu^2\exp\left(\frac{u_\ep^2-u_\ep^1}{\ep}\right)-\nu^1 \right)^2\leq  \ep d\,\nabla R^1(x_\ep,I_\ep^1)\cdot \nabla u_\ep^1(x_\ep).$$
From \ref{as:quad1}, \ref{as:quad2} and \ref{boundI} we find that $(R^1(x,I_\ep^1))_\ep$ is uniformly bounded for $\ep\leq \ep_0$. We conclude that $(u_\ep^1)_\ep$ is uniformly Lipschitz  for $\ep\leq \ep_0$.\\

To prove uniform bounds from below, we notice from \ref{as:psi} and \ref{boundI} that, for $i=1,2$, there exists a point $\overline x_i\in B_L(0)$ such that
$$\ep \ln\left(\frac{I_m}{a_M|B_L(0)|}\right)\leq u_\ep^i(\overline x_i).$$
From the latter and the Lipschitz bounds we obtain that
$$-2LC_1+\ep \ln\left(\frac{I_m}{a_M|B_L(0)|}\right)\leq u_\ep^i,\qquad \text{in $B_L(0)$ and for $i=1,2$}.$$
It follows that the families $(u_\ep^i)_\ep$ are bounded from below for $\ep\leq \ep_0$ and $i=1,2$.

(iii) We prove \ref{lbound} for $i=1$ by contradiction. The proof for $i=2$ follows the same arguments. We assume that there exists a sequence $(\ep_k,x_k)$ such that $\ep_k\to 0$ as $k\to \infty$, $x_k\in B_L(0)$ and $u_{\ep_k}^1(x_k)>a$. Using the uniform Lipschitz bounds obtained in (ii) we have
$$n_{\ep_k}^1(x)>\exp\left(\frac{a}{2\ep_k}\right),\qquad \text{in }[x_k-\frac{a}{2C_1},x_k+\frac{a}{2C_1}]\cap B_L(0).$$
This is in contradiction with the bound from above in \ref{boundI}, for $\ep_k$ small enough. Therefore \ref{lbound} holds.
\end{proof}

\section{Convergence to the Hamilton-Jacobi equation}\label{sec:HJ}

In this section we prove Theorem \ref{thm:main}. \\

\begin{proof}
\textbf{Convergence to the Hamilton-Jacobi equation: }
From (ii) and (iii) in Theorem \ref{thm:reg} we have that for $i=1,2$, the families $(u_\ep^i)_\ep$ are uniformly bounded and Lipschitz. Therefore, from the Arzela-Ascoli Theorem we obtain that, along subsequences, $(u_\ep^1)_\ep$ and $(u_\ep^2)_\ep$ converge locally uniformly to some continuous functions $u^i\in \mathrm{C}(B_L(0); \mathbb{R})$, with $i=1,2$. Moreover, from (i) in Theorem \ref{thm:reg} we deduce that $u^1=u^2$. Here we consider a subsequence of $(I_\ep^1,I_\ep^2,u_\ep^1,u_\ep^2)_\ep$ that converges to $(I^1,I^2,u,u)$.

Let $H(x,I_\ep^1,I_\ep^2)$, be the largest eigenvalue of the matrix
$$\mathcal{A}_\ep=\left(\begin{array}{cc} R^1(x,I_\ep^1)-\nu^1&\nu^2\\
\nu^1 & R^2(x,I_\ep^2)-\nu^2\end{array}\right),$$
and $\left(\begin{array}{c}\chi_\ep^1(x)\\\chi_\ep^2(x)\end{array}
\right)$ be the corresponding eigenvector. Since the non-diagonal terms in $\mathcal{A}_\ep$ are strictly positive, using the Perron-Frobinius Theorem, we know that such eigenvalue exist and that $\chi_\ep^1$ and $\chi_\ep^2$ are strictly positive. We write
$$\phi_\ep^i(x)=\ln \chi_\ep^i(x),\qquad \text{for $i=1,2$.}$$
 We prove that $u$ is a viscosity solution of 
$$ - |\nabla u | ^2 = H(x,I^1,I^2), \qquad  \text{in $B_L(0)$}. $$
To this aim, suppose that $u-\varphi$ has a maximum in $x\in B_L(0)$. Then, we consider a sequence $x_\ep\in B_L(0)$, such that as $\ep\to 0$, $x_\ep\to x$ and 
$$u_\ep^1(x_\ep)-\varphi(x_\ep)-\ep\phi_\ep^1(x_\ep)=\max_{\underset{i=1,2}{x\in B_L(0)}}  u_\ep^i(x)-\varphi(x)-\ep\phi_\ep^i(x)$$
is attained at the point $x_\ep$ and for $i=1$ (The case with $i=2$ can be treated similarly). In this case, we have in particular that
$$u_\ep^2(x_\ep)-u_\ep^1(x_\ep)\leq \ep\left(\phi_\ep^2(x_\ep)-\phi_\ep^1(x_\ep)\right).$$
Using the latter and the viscosity criterion for the first equation in \ref{equ} we obtain that
\begin{equation}\label{ine}\begin{split}
-\ep (\Delta \varphi(x_\ep)+\ep\Delta \phi_\ep^1(x_\ep))-|\nabla \varphi(x_\ep)+\ep\nabla \phi_\ep^1(x_\ep)|^2- R^1(x,I_\ep^1)&\\
-\nu^2 \exp\left(\phi_\ep^2(x_\ep)-\phi_\ep^1(x_\ep)\right)+\nu^1&\leq 0.
\end{split}
\end{equation}
We notice that, by definition of $\phi_\ep^1$ and $\phi_\ep^2$, we have
$$- R^1(x,I_\ep^1)-\nu^2 \exp\left(\phi_\ep^2(x_\ep)-\phi_\ep^1(x_\ep)\right)+\nu^1=-H(x,I_\ep^1,I_\ep^2).$$
From the latter and by letting $\ep\to 0$ in \ref{ine} we deduce that
$$
-|\nabla \varphi(x)|^2\leq H(x,I^1,I^2),$$
and thus $u$ is a subsolution of \ref{HJ} in the viscosity sense. The supersolution criterion can be proved in a similar way.\\

\textbf{The constraint on the limit ($\max_{x\in B_L(0)}u(x)=0$):} From \ref{lbound} we obtain that $u(x)\leq 0$. To prove that $0\leq\max_{x\in B_L(0)}u(x)$, we use the lower bounds on $I_\ep^i$ in \ref{boundI}.  The proof of this property is classical and we refer to \cite{GB.SM.BP:09,AL.SM.BP:10} for a detailed proof.
\end{proof}

\section{Asymptotic behavior of stationary solutions}\label{sec:as}

In this section we prove Theorem \ref{thm:as}.

\begin{proof}
\textbf{Support of $n^i$:} From \ref{boundI}, we deduce that, along subsequences and for $i=1,2$, $(n_\ep^i)_\ep$ converges weakly to a measure $n^i$. The fact that $\mathrm{supp } \;n^i\subset \Omega,\text{ for $i=1,2$,}$ is a consequence of the Hopf-Cole transformation \ref{Hopf}. To prove \ref{eq:supp} it is enough to prove $\Omega \subset \Gamma$. To this aim following the idea in \cite{GB.BP:08} we first prove that, for $i=1,2$, $u_\ep^i$ are uniformly semi-convex. Recall that the smooth function $v$ is semiconvex with constant $C$, if we have
$$v_{\xi\xi}\geq -C,\qquad \text{for all $|\xi|=1$}.$$
Let 
\begin{equation}\label{min}\min \{u_{\ep,\xi\xi}^i(x)\,|\, x\in B_L(0),\, i=1,2,\, \xi\in \mathbb{R}^d, \, |\xi|=1\}=u_{\ep,\eta\eta}^1(x_\ep).\end{equation}
The case where the minimum is achieved for $i=2$ can be treated similarly. We differentiate twice the first equation in \ref{equ} with respect  to $\eta$ and obtain
\begin{equation*}\begin{split}
-\ep \Delta u_{\ep,\eta\eta}^1&=2\nabla u_\ep^1\cdot \nabla u_{\ep,\eta\eta}^1+2|\nabla u_{\ep,\eta}^1|^2+R_{\eta\eta}^1
\\
&+\nu^2\left(\left(\frac{u_{\ep,\eta}^2-u_{\ep,\eta}^1}{\ep}\right)^2+\frac{u_{\ep,\eta\eta}^2-u_{\ep,\eta\eta}^1}{\ep}\right)\exp\left(\frac{u_\ep^2-u_\ep^1}{\ep}\right).
\end{split}
\end{equation*}
From \ref{min} we obtain that  $\Delta u_{\ep,\eta\eta}^1(x_\ep)\geq 0$, $\nabla u_{\ep,\eta\eta}^1(x_\ep)=0$ and $u_{\ep,\eta\eta}^2(x_\ep)-u_{\ep,\eta\eta}^1(x_\ep)\geq 0$. Using \ref{as:R2} It follows that
$$|\nabla u_{\ep,\eta}^1(x_\ep)|^2\leq \frac D 2.$$
Since $u_{\ep,\eta\eta}^1=\nabla u_{\ep,\eta}^1\cdot \eta$, we have $|u_{\ep,\eta\eta}^1|\leq |\nabla u_{\ep,\eta}^1|$. We deduce that
$$|u_{\ep,\eta\eta}^1(x_\ep)|^2\leq \frac D 2,$$
and thus
$$\min \{u_{\ep,\xi\xi}^i(x)\,|\, x\in B_L(0),\, i=1,2,\, \xi\in \mathbb{R}^d, \, |\xi|=1\}\geq -\sqrt{ \frac D 2}.$$
This proves that $u_\ep^i$, for $i=1,2$ are semiconvex functions with constant $-\sqrt{ \frac D 2}$. By passing to the limit in $\ep\to 0$ we obtain that $u$ is also semiconvex with the same constant. \\

A semiconvex function is differentiable at its maximum points. Therefore $u$ is differentiable with $\nabla u=0$ in the set $\Omega$. From \ref{HJ}, we deduce, that for all $x\in \Omega$, $H(x,I^1,I^2)=0$, and thus $\Omega \subset \{x\in B_L(0)\,|\,H(x,I^1,I^2)=0\}$. The fact that $\max_{x\in B_L(0)}H(x,I^1,I^2)=0$ is immediate from \ref{HJ} and the facts that $u$ is almost everywhere differentiable and $H(x,I^1,I^2)$ is a continuous function.  \\

\textbf{Value of $n^i$ on the support: }  Let $\xi\in \mathrm{C}^\infty_\mathrm{c}(B_L(0))$, i.e. $\xi$ is a smooth function with compact support in $B_L(0)$. We multiply \ref{stat} by $\xi$ and integrate with respect to $x$ in $B_L(0)$ to obtain, for $ \{i,j\}=\{1,2\}$,
\begin{equation*}\begin{split}- \ep^2 \int_{B_L(0)} n_\ep^i(x) \Delta \xi (x) dx   &=\int_{B_L(0)} \xi(x) n_\ep^i(x) R^i(x,I_\ep^i) dx\\
& -\nu^i \int_{B_L(0)} \xi(x) n_\ep^i(x)dx +\nu^j \int_{B_L(0)} \xi(x)n_\ep^j(x)  dx. \end{split}
\end{equation*}
Since  $n_\ep^l\xrightharpoonup{\quad} n^l$ weakly and $I_\ep^l\to I^l$, for $l=1,2$, as $\ep\to 0$, we obtain that, for $ \{i,j\}=\{1,2\}$,
$$\int_{B_L(0)} \xi(x) n^i(x) R^i(x,I^i) dx -\nu^i \int_{B_L(0)} \xi(x) n^i(x)dx +\nu^j \int_{B_L(0)} \xi(x)n^j(x)  dx=0,$$
and thus \ref{dist}.  Finally, \ref{eq:I1I2first} follows from  \ref{Iei}.
\end{proof}

\section{Examples of application}\label{sec:ex}

In \ref{eq:ni}--\ref{eq:I1I2} we give a description of  $(n^1,n^2)$, assuming that the support of $n^i$, for $i=1,2$, is  a set of distinct points, i.e. $n^i$ is a  sum of Dirac masses and does not have a continuous distribution. This is what we expect naturally in the models based on darwinian evolution. More precisely, from Volterra-GauseÕs competitive exclusion principle (see \cite{Levin_70,Schoener:74}) it is known in theoretical biology that in a model with $K$ limiting factors (as nutrients or geographic parameters) at most $K$ distinct species can generally survive. Here we have two limiting factors, represented by $I^1$ and $I^2$, that correspond to the environmental pressures in the two patches. We thus expect to observe only monomorphic or dimorphic situations. This is also the case in the numerical simulations represented in Section \ref{sec:num}.

From \ref{eq:supp} we know that the support of $n^i$ is included in the set of maximum points of $H(x,I^1,I^2)$, $\Gamma$, with $(I^1,I^2)$ limits of $(I_\ep^1,I_\ep^2)$. If now $H$ is such that, for fixed $(I^1,I^2)$, the corresponding set $\Gamma$ consists of isolated points, it follows that the supports of $n^1$ and $n^2$ consist also of isolated points. We give an example below where $H$ has clearly this property.\\

\begin{example} \label{ex1}(\textbf{monomorphism towards dimorphism})
Consider a case with the following values for the parameters of the system
\begin{equation}\label{quad}R^1(x,I)=a^1x^2+b^1x+c^1-d^1I,\qquad R^2(x,I)=a^2x^2+b^2x+c^2-d^2I,\end{equation}
with 
$$a^i,b^i,c^i,d^i\in \mathbb{R}, \qquad a^i<0<d^i, \qquad \text{for $i=1,2$.}$$
Then the supports of $n^1$ and $n^2$ consist at most of two single points.
\end{example}

We first notice that in the case where there is no migration between patches ($\nu^1=\nu^2=0$), from the results in \cite{AL.SM.BP:10}, we know that in patch $i$, the population concentrates in large time on the maximum points of $R^i(\cdot,I^i)$ with $I^i$ the limit of $I_\ep^i$. Since $R^i$ is a quadratic function in $x$, it has a unique maximum and thus  $n^i$ is a single Dirac mass on this maximum point. However, allowing migration by  taking positive values for $\nu^1$ and $\nu^2$ the population can become dimorphic. In Section \ref{sec:num} we give a numerical example where a dimorphic situation appears (see Figure \ref{fig:dim}). This is in accordance with the competitive exclusion principle since we have introduced a new limiting factor, which is the choice of habitable zones. 

Next, we prove the result:\\

\begin{proof}[Proof of Example \ref{ex1}. ] From \ref{eq:supp} we have that the stationary solutions concentrate asymptotically on the maximum points of $H$ defined as below
$$\begin{array}{rl}H(x,I^1,I^2)&=\frac 12F+\frac 12\sqrt{F^2-4G},\end{array}$$
with 
\begin{equation}
\label{def:G}
\begin{array}{c}F(x,I^1,I^2):=R^1(x,I^1)-\nu^1+R^2(x,I^2)-\nu^2,\\ G(x,I^1,I^2):=(R^1(x,I^1)-\nu^1)(R^2(x,I^2)-\nu^2)-\nu^1\nu^2,
\end{array}\end{equation}
Since $\max_{x\in B_L(0)}H(x,I^1,I^2)=0$, we deduce that
\begin{equation}
\label{G:min}\min_{x\in B_L(0)}G(x,I^1,I^2)=0,\end{equation}
and
\begin{equation}\label{HG}\Gamma=\{x\in B_L(0)\,|\, H(x,I^1,I^2)=0\}=\{x\in B_L(0)\,|\, G(x,I^1,I^2)=0\}.\end{equation}
For fixed $(I^1,I^2)$, $G(x,I^1,I^2)$ is a polynomial of order $4$. Therefore it has at most two maximum points. It follows that $\Gamma$ consists of one or two distinct points. 
%
\end{proof}
\vspace{2mm}

\begin{example}\label{ex2}(\textbf{An asymmetric case}) We assume that the parameters are such that the support of $n^i$, for $i=1,2$, consists of isolated points, and we have
\begin{equation}\label{as:sym}\begin{split}
R^i(x,I)=R^i(x)-cI,\quad \text{for $i=1,2$,} \quad R^1(x)=R^2(\tau(x)),\quad \text{for all $x\in B_L(0)$}\\
\text{ and }\quad \nu^1=\nu^2=\nu,
\end{split}
\end{equation}
with $\tau:B_L(0)\to B_L(0)$ such that $\tau \circ  \tau=\rm{Id}$.
Let  $(I^1,I^2)$ be a limit point of $(I_\ep^1,I_\ep^2)$.  We have $I^1=I^2=I$, where $I$ is such that
$$\max_x H(x,I,I)=\min_x G(x,I,I)=0,$$
with $H$and $G$ defined respectively in \ref{def:H} and \ref{def:G}.
 In particular, if $\bar x\in \Gamma$ then we have $\tau(\bar x)\in \Gamma$, with $\Gamma$ defined in \ref{Gamma}.
\end{example}

Assumption \ref{as:sym} covers the case where the growth terms have the following forms
$$R^1(x)=f(|x-a|),\qquad R^2(x)=f(|x+a|),$$ 
with $f:\rm{B_L(0)}\to \mathbb{R}$ a function and $a$ constant (we consider the application $\tau(x)=-x$).  In this case the competition terms in the patches have a simple form: the fitness, in absence of migration, has a shift in traits from one zone to another, for instance due to a difference in the temperature. We can thus characterize the limit in this case. If moreover, we suppose that the growth terms satisfy \ref{quad}, we conclude that in the limit while $\ep\to 0$, the population, is either monomorphic with a single Dirac mass at the origin, or it is dimorphic with two Dirac masses located on two symmetric points, one of the winning traits being more favorable for zone $1$ and the other one being more favorable for zone $2$.
\\

\begin{proof}[Proof of Example \ref{ex2}.]
We prove the claim by contradiction and we assume that $I^1\neq I^2$. Without loss of generality we suppose that $I^1< I^2$. Let $\bar x_j\in \rm{supp}\;n^1$.  From \ref{eq:supp} and \ref{HG}, we have that $G$ has a minimum in $\bar x_j$ and in particular, $G(\bar x_j)\leq G(\tau(\bar x_j))$, namely,
$$(R^1(\bar x_j)-I^1-\nu)(R^2(\bar x_j)-I^2-\nu)\leq (R^1(\tau(\bar x_j))-I^1-\nu)(R^2(\tau(\bar x_j))-I^2-\nu).$$
It follows that
$$(R^1(\bar x_j)-I^1-\nu)(R^2(\bar x_j)-I^2-\nu)\leq (R^2(\bar x_j)-I^1-\nu)(R^1(\bar x_j)-I^2-\nu).$$
We deduce that
$$0\leq (I^1-I^2)\left(R^2(\bar x_j)-R^1(\bar x_j)\right), $$
and thus 
$$R^2(\bar x_j)\leq R^1(\bar x_j).$$
From the latter, $I^1<I^2$ and \ref{eq:rho} we obtain that
$$\rho_j^2<\rho_j^1. $$
Since this is true for all $\bar x_j\in \rm{supp}\; n^1= \rm{supp}\; n^2$, we obtain from \ref{eq:I1I2} that $I^2<I^1$. This is a contradiction and thus $I^1=I^2$.
\end{proof}
\section{The case with several patches}\label{sec:ext}

The result can be extended to the case with more than two patches. The model for $K$ patches is written as
\begin{equation}\label{multistat}\left\{\begin{array}{rll}
-\ep^2 \Delta n_\ep^i&=n_\ep^i R^i(x,I_\ep^i)+\sum_j \nu^{ij} n_\ep^j-\nu^{ii} n_\ep^i&\quad \text{in $B_L(0)$ and for $1\leq i \leq K$},\\\\
\nabla n_\ep^i\cdot \vec{n}&=0&\quad \text{in $\partial B_L(0)$ and for $1\leq i \leq K$},
\end{array}\right.
\end{equation}
with 
\begin{equation}\label{multiIei}
I_\ep^i=\int \psi^i(x)n_\ep^i(x)dx,\qquad \text{for $i\in \{1, \cdots ,K\}$}.
\end{equation}
We suppose that $(n_\ep^1,\cdots, n_\ep^K)$ is a solution of \ref{multistat}--\ref{multiIei} such that
\begin{equation}\label{as:multi}
\min(I_\ep^1,\cdots, I_\ep^K)\leq I_M,\quad I_m\leq \max(I_\ep^1,\cdots,I_\ep^K),\quad (I_\ep^1,\cdots,I_\ep^K)\underset{\ep\to 0}{\longrightarrow}(I^1,\cdots,I^K).
\end{equation}
We also replace assumption \ref{as:R2} by
\begin{equation}\label{as:multiR}\begin{array}{c}
|R^i(x,I)|\leq C,\quad -D\leq R_{\xi\xi}^i(x,I), \\
\text{for $x\in B_L(0)$, $0\leq I$, }  \text{$\xi\in \mathbb{R}^d$, $|\xi|=1$ and $1\leq i\leq K$,}
\end{array}
\end{equation}
and we use again the Hopf-Cole transformation
$$n_\ep^i=\exp(\frac{u_\ep^i}{\ep}),\qquad \text{for $i=1,\cdots,K$.}$$
To present the result we also introduce the following matrix
$$\mathcal{B}=\left(\begin{array}{ccc} R^1(x,I^1)-\nu^{11}&\cdots&\nu^{1K}\\
\vdots&\ddots&\vdots\\
\nu^{K1} &\cdots &R^K(x,I^K)-\nu^{KK}\end{array}\right),$$
and as in the case with two patches we define
$$
\Omega=\{x\in B_L(0) \,| \, u(x)=0\},$$
and
$$
\Gamma=\{x\in  B_L(0)\,| \, H(x,I^1,\cdots,I^K)=\max_{x\in B_L(0)}H(x,I^1,\cdots,I^K)=0\}.
$$

We have
\begin{theorem}\label{thm:multimain}
We assume that  $(n_\ep^1,\cdots, n_\ep^K)$ is a solution of \ref{multistat}--\ref{multiIei} with \ref{as:psi}, \ref{as:multi} and \ref{as:multiR}. Then, after extraction of a subsequence, the sequences $(u_\ep^i)_\ep$, for $i=1,\cdots,K$, converge to a continuous function $u\in \mathrm{C}(B_L(0))$ that is a viscosity solution to the following equation
$$
\left\{ \begin{array}{ll}-|\nabla u|^2= H(x,I^1,\cdots,I^K),&\quad \text{in $B_L(0)$},\\\\
\max_{x\in B_L(0)}u(x)=0,\end{array}\right.$$
with 
$H(x,I^1,I^2)$ the largest eigenvalue of the matrix $\mathcal{B}$. 
Let $n^i$, for $i=1,\cdots,K$, be a weak limit of $n_\ep^i$ along this subsequence.  We have
$$
\mathrm{supp } \;n^i\subset \Omega\cap \Gamma,\quad \text{for $i=1,\cdots,K$}$$
Moreover, if the support of $n^i$, for $i=1,\cdots,K$, is a set of distinct points: $\mathrm{supp}\,n^i\subset \{x_1,x_2,\cdots,x_l\}$, we then have
$$
n^i=\sum_{j=1}^l\rho^i_j \delta(x-x_j),\qquad \text{for $i=1,\cdots,K$,}$$
with $\left(\begin{array}{c}\rho^{1}_j\\\vdots\\ \rho^K_j\end{array}\right)$ the eigenvector corresponding to the largest eigenvalue of $\mathcal{B}$ at the point $x_j$, which is $0$, 
coupled by
$$\sum_j \rho_j^i\psi^i(x_j)=I^i.$$

\end{theorem}
\begin{proof}
The proof of Theorem \ref{thm:multimain} follows along the same lines as the one of Theorem \ref{thm:main} and Theorem \ref{thm:as}. The only difference is in the proof of lower and upper bounds on $u_\ep$ which are obtained using the uniform bounds on $I_\ep^i$. Indeed Assumption \ref{as:multi} is slightly weaker than \ref{boundI}. To prove uniform bounds on $u_\ep^i$, with $i=1,\cdots,K$, using \ref{as:multi} we first prove that for an index $j\in \{1,\cdots,K\}$ which is such that the minimum (respectively the maximum) of $(I^1,\cdots,I^K)$ is attained for  $I^j$, $u_\ep^j$ is uniformly bounded from above (respectively from below), then we use an estimate of type \ref{harnack} to obtain a uniform bound from above (respectively from below) on $u_\ep^i$ for all $i\in\{1,\cdots,K\}$.
\end{proof}

\section{Time dependent problem and numerics}\label{sec:num}

How well the asymptotics of the solutions of \ref{stat} (that are stationary solutions of \ref{main}) approximate the large time behavior of the solution of the time-dependent problem \ref{main}, while $\ep$ vanishes ? In this section, using numerical simulations we try to answer to this question. Theoretical study of the time-dependent problem, which requires appropriate regularity estimates, is beyond the scope of the present paper and is left for future work.

The numerical simulations for \ref{main} have been performed in Matlab using the following parameters
\begin{equation}\label{data1}
\begin{array}{c}R^1(x,I)=3-(x+1)^2-I,\quad R^2(x,I)=3-(x-1)^2-I,\quad\psi^1(x)=\psi^2(x)=1,\\\ \nu^1=\nu^2=2.5,\quad {\ep=.001},\quad L=2.
\end{array}
\end{equation}
We notice that these parameters verify the properties in both examples \ref{ex1} and \ref{ex2}. Therefore, we expect that the stationary solutions are concentrated on one or two Dirac masses that are symmetric with respect to the origin. As we observe in Figure \ref{fig:mon},  $n_\ep^1$ and $n_\ep^2$,  with $(n_\ep^1,n_\ep^2)$ solution of the time-dependent problem \ref{main}  with the above parameters, concentrate in large time on a single Dirac mass at the origin, which is the mean value of the favorable traits in each zone in absence of migration. In this simulation, initially  $n_\ep^1$ is concentrated on $x=-0.3$ and $n_\ep^2$ is concentrated on $x=0.3$.
\begin{figure}[h]
\includegraphics[angle=0,scale=0.4]{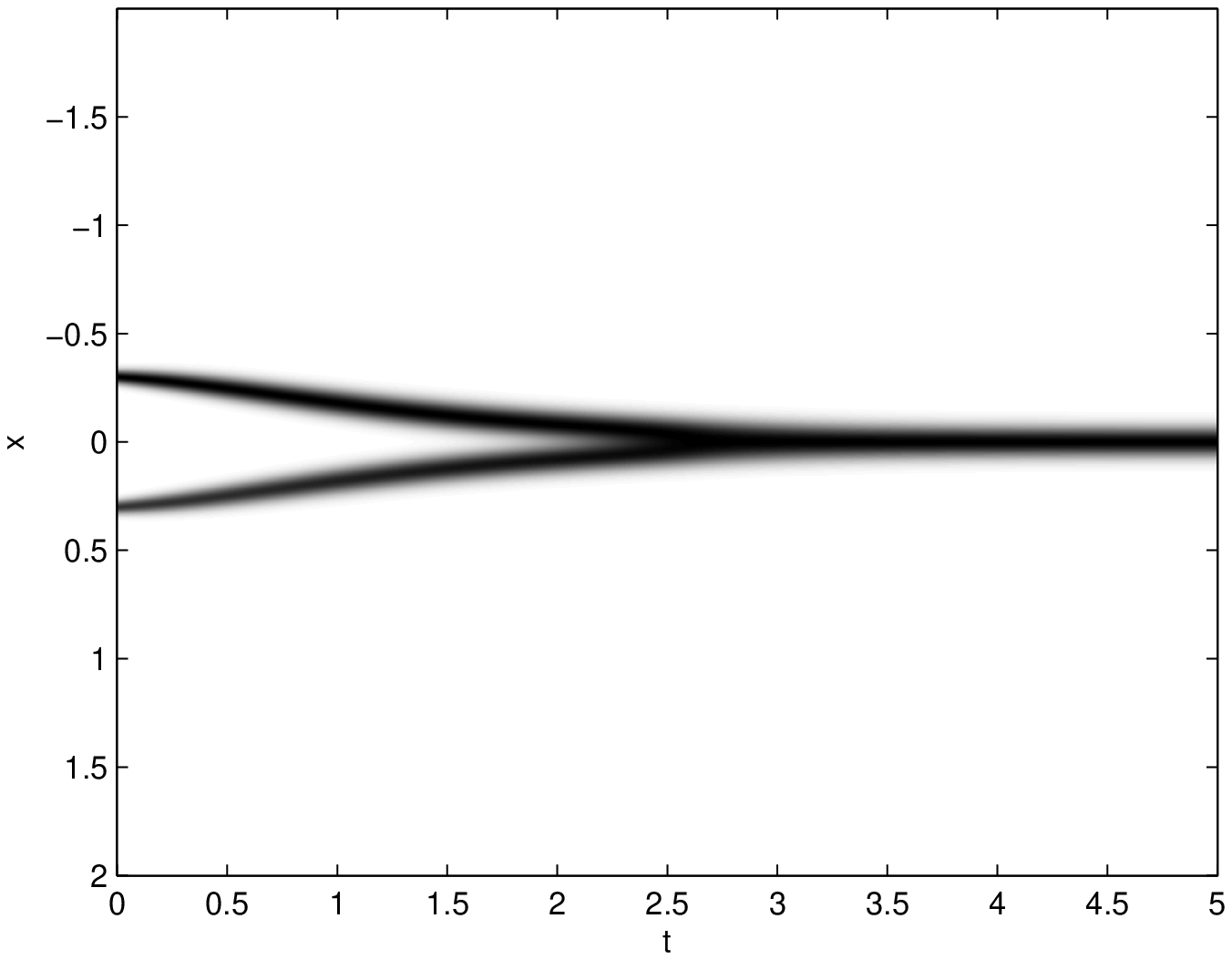}\includegraphics[angle=0,scale=0.4]{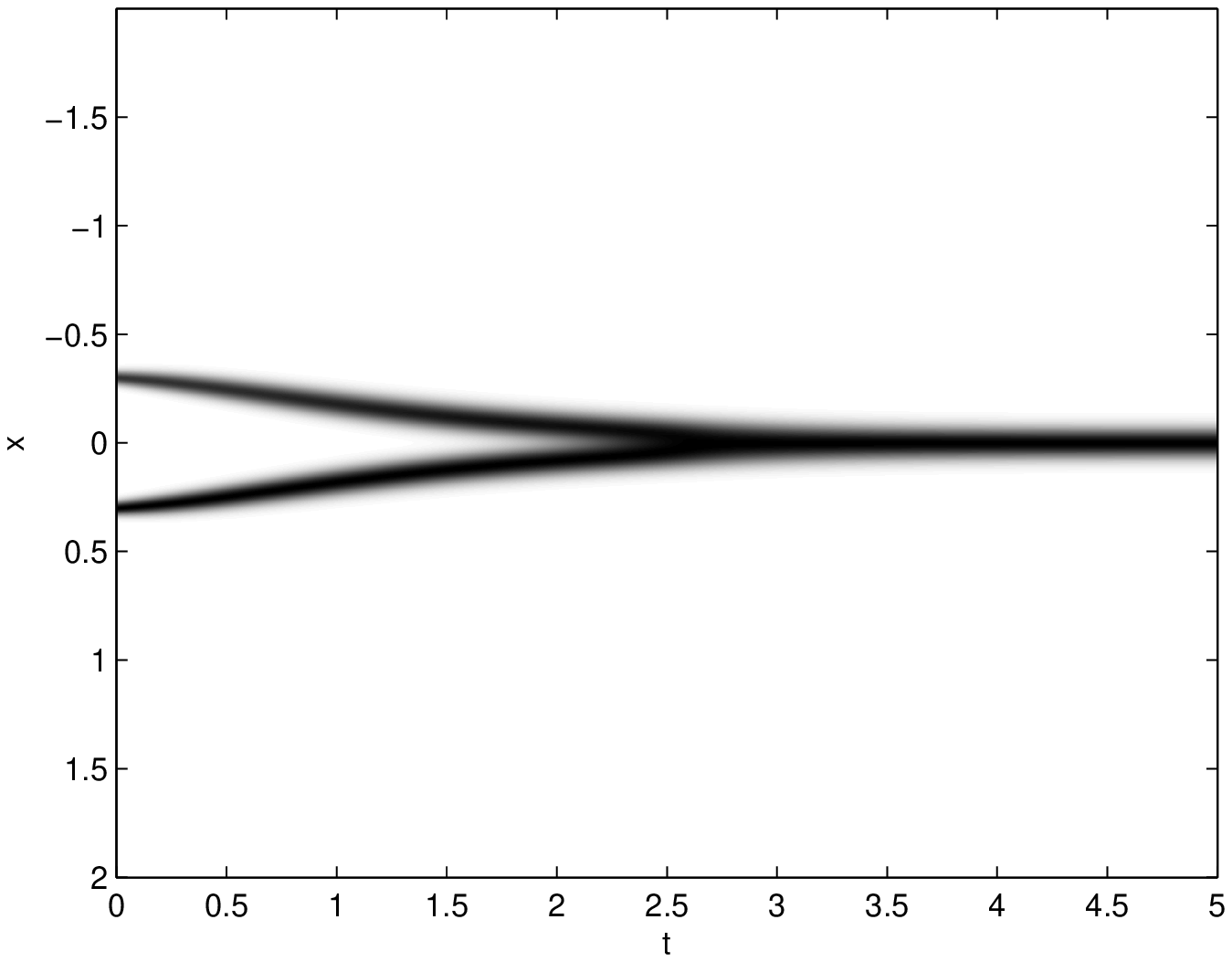}
\caption{Dynamics of the time-dependent problem \ref{main} with parameters given in \ref{data1}. In both figures, horizontally is time $t$  and vertically is trait $x$. The gray layers represent the value of $n_\ep^1$(left) and $n_\ep^2$(right). Initially  $n_\ep^1$ is concentrated on $x=-0.3$ and $n_\ep^2$ is concentrated on $x=0.3$. Due to migration both traits appear rapidly in the two patches, but in large time only one dominant trait persists. This point is the mean value of favorable traits in each patch in absence of migration. }
\label{fig:mon}
\end{figure}
Depending on the parameters of the model, one can also observe stability in large time of dimorphic situations. For instance, if we vary the values of $\nu^1$ and $\nu^2$ in \ref{data1} as follows 
\begin{equation}\label{data2}
\begin{array}{c}
R^1(x,I)=3-(x+1)^2-I,\quad R^2(x,I)=3-(x-1)^2-I,\quad \psi^1(x)=\psi^2(x)=1\\
 \nu^1=\nu^2=1,\quad {\ep=.001},\quad L=2,\end{array}
\end{equation}
then $n_\ep^1$ and $n_\ep^2$, with $(n_\ep^1,n_\ep^2)$ solution of the time-dependent problem \ref{main},  concentrate in large time  on two distinct Dirac masses, one of them more favorable to patch $1$ and the second one more favorable to patch $2$ (see Figure \ref{fig:dim}). We note indeed that, in absence of migration, the local optimal trait in patch $1$ is $x=-1$ and in patch $2$ is $x=1$. In presence of migration, the two initial traits appear immediately in the two patches and evolve to two points, one close to $x=-0.86$ and the other close to $x=0.86$.
\begin{figure}[h]
\includegraphics[angle=0,scale=0.4]{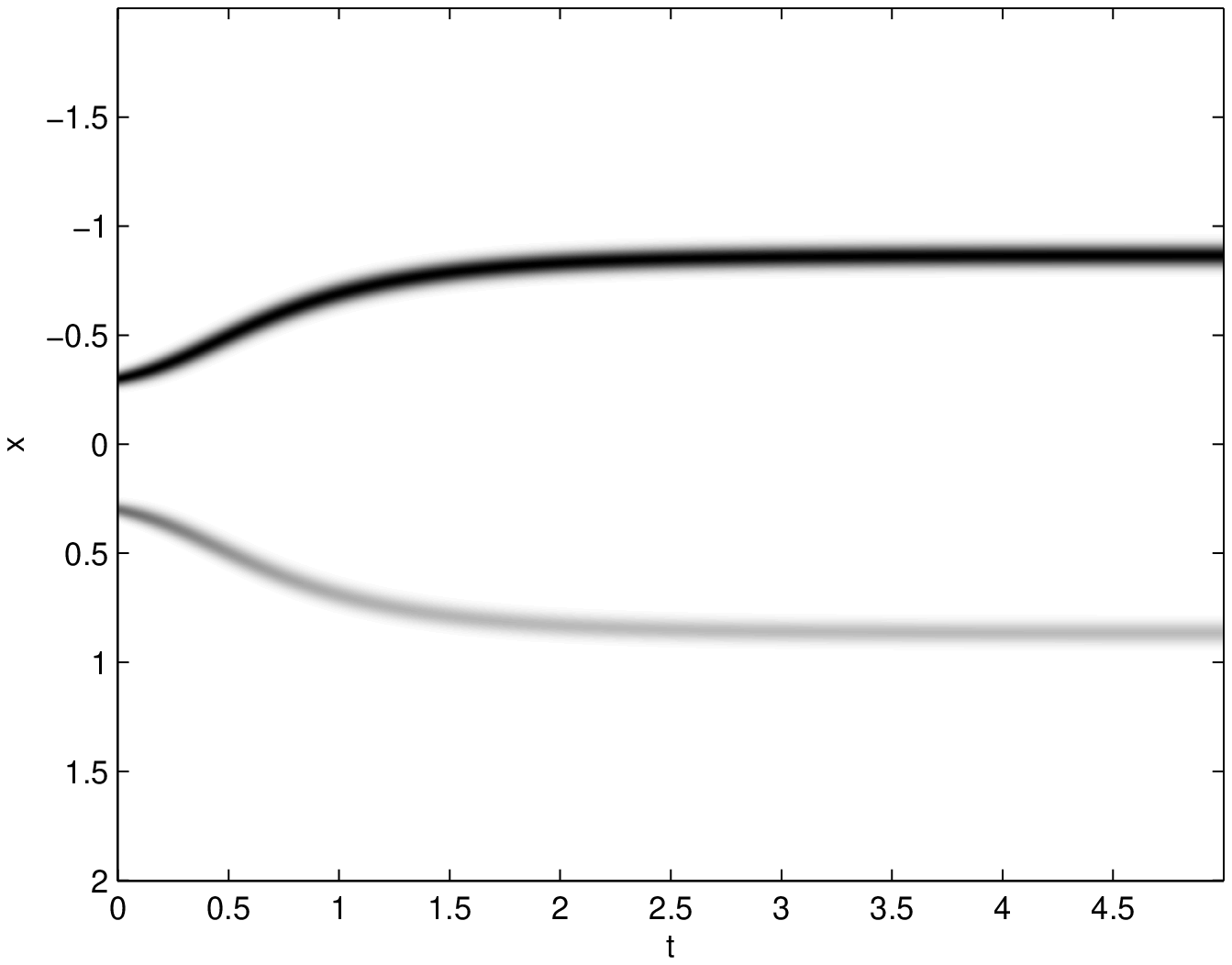}\includegraphics[angle=0,scale=0.4]{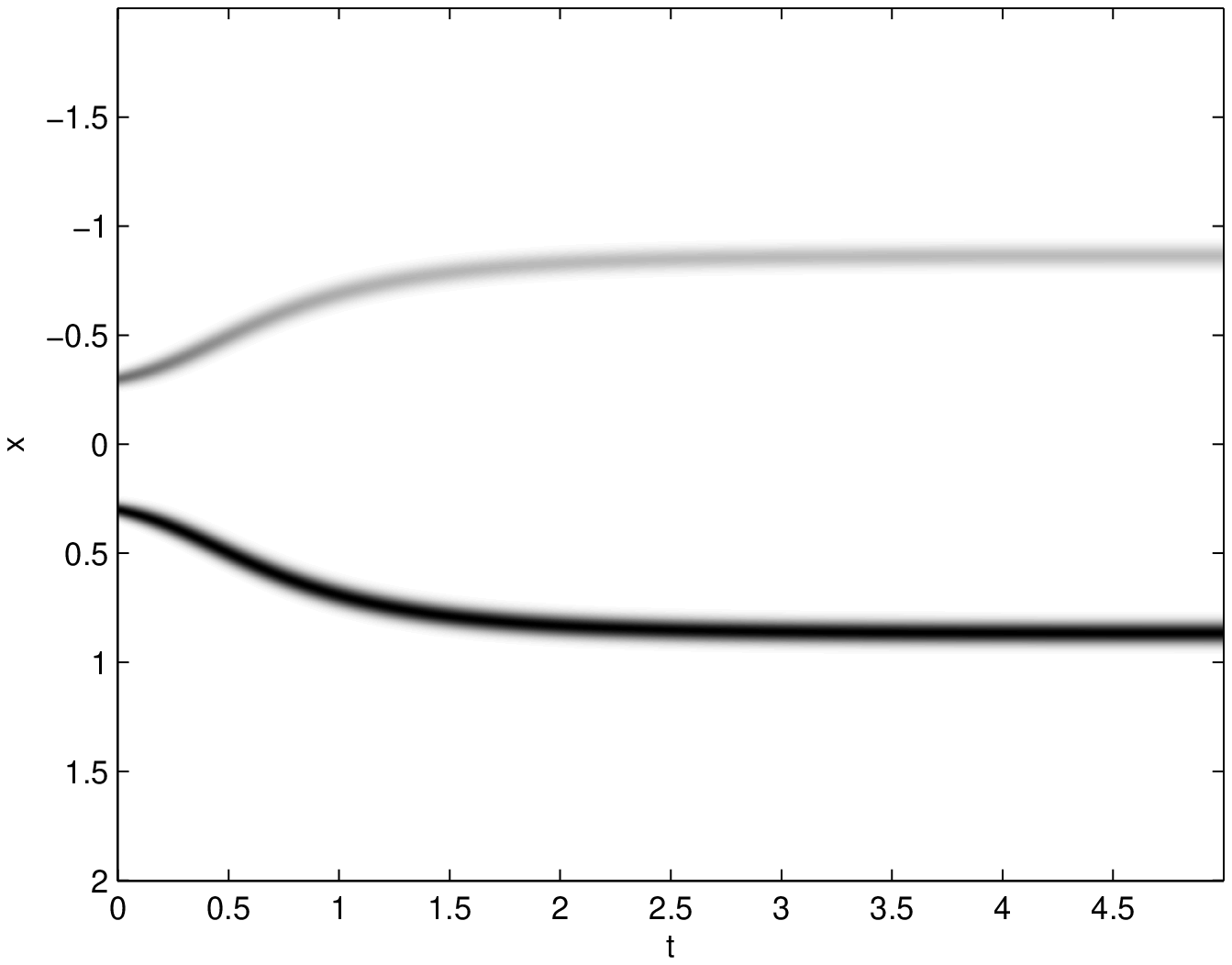}
\caption{Dynamics of the time-dependent problem \ref{main} with parameters given in \ref{data2}. In both figures, horizontally is time $t$ and vertically is trait $x$.  The gray layers represent the value of $n_\ep^1$(left) and $n_\ep^2$(right). In absence of migration, the local optimal trait in patch $1$ is $x=-1$ and in patch $2$ is $x=1$. Initially  $n_\ep^1$ is concentrated on $x=-0.3$ and $n_\ep^2$ is concentrated on $x=0.3$.  Due to migration both traits appear rapidly in the two patches, and evolve to two points close to $ -0.86$ and $0.86$.  }
\label{fig:dim}
\end{figure}

Does the above numerical solution converge in long time to the solution described by the algebraic constraints given in Theorem \ref{thm:as}?  The values of $I_\ep^1$ and $I_\ep^2$ are depicted in Figure \ref{fig:per} showing that both these quantities converge in long time to $2.25$. We can also compute the value of $H$ at the final time step. As we observe in Figure \ref{fig:H}, $\max_x H=0$ and the maximum is attained at the points $x=-.86$ and $x=.86$ which correspond to the positions of the Dirac masses in Figure \ref{fig:dim}. We can also compute numerically the weights of the Dirac masses at the final time step, to obtain 
$$n_\ep^1(t=5)\approx 1.77\, \delta(x+.86)+.48\, \delta(x-.86),\quad n_\ep^2(t=5)\approx.48\, \delta(x+.86)+1.77\, \delta(x-.86).$$
One can verify that the above weights satisfy \ref{eq:rho}--\ref{eq:I1I2}.
   
\begin{figure}[h]
\includegraphics[angle=0,scale=0.4]{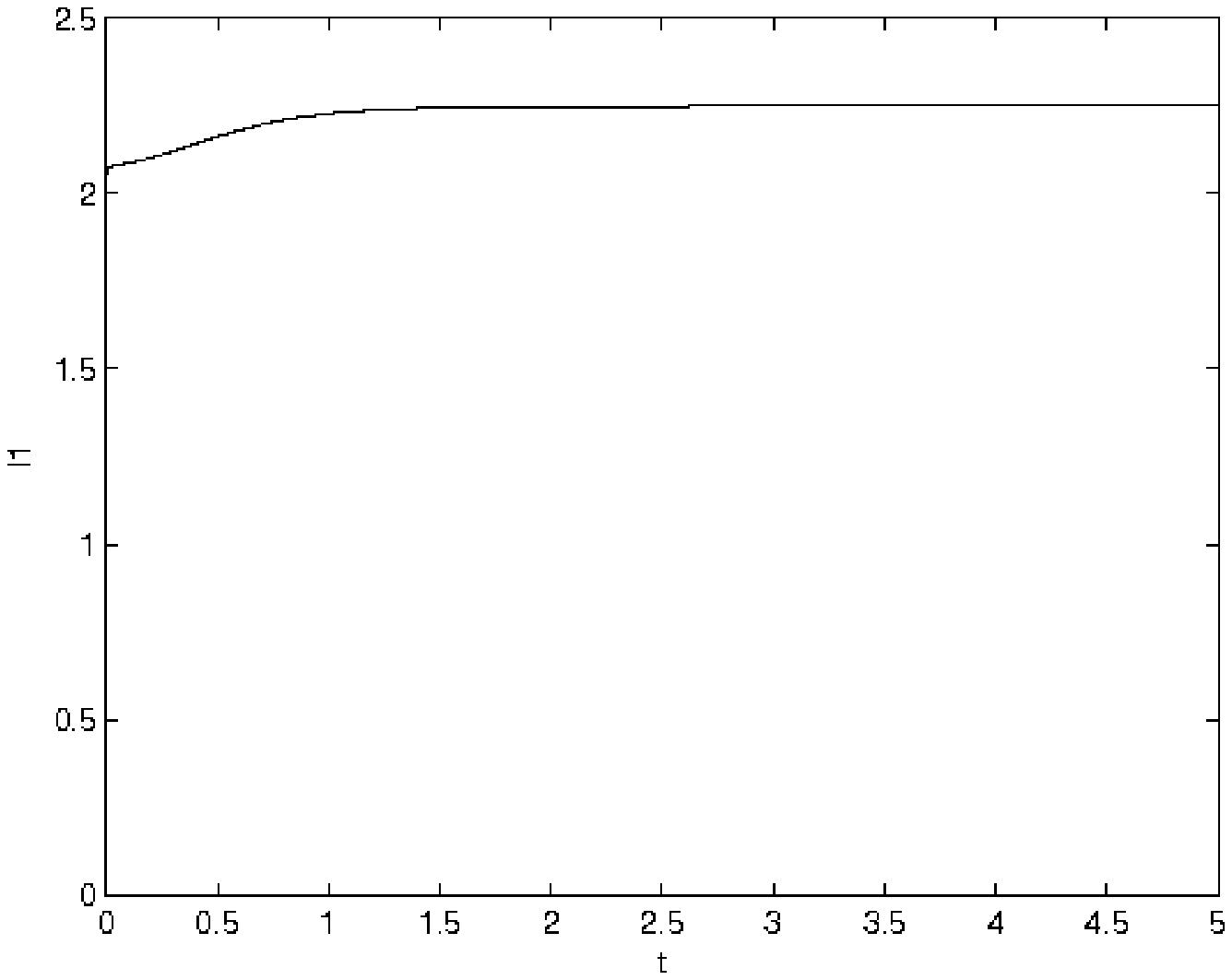}\includegraphics[angle=0,scale=0.4]{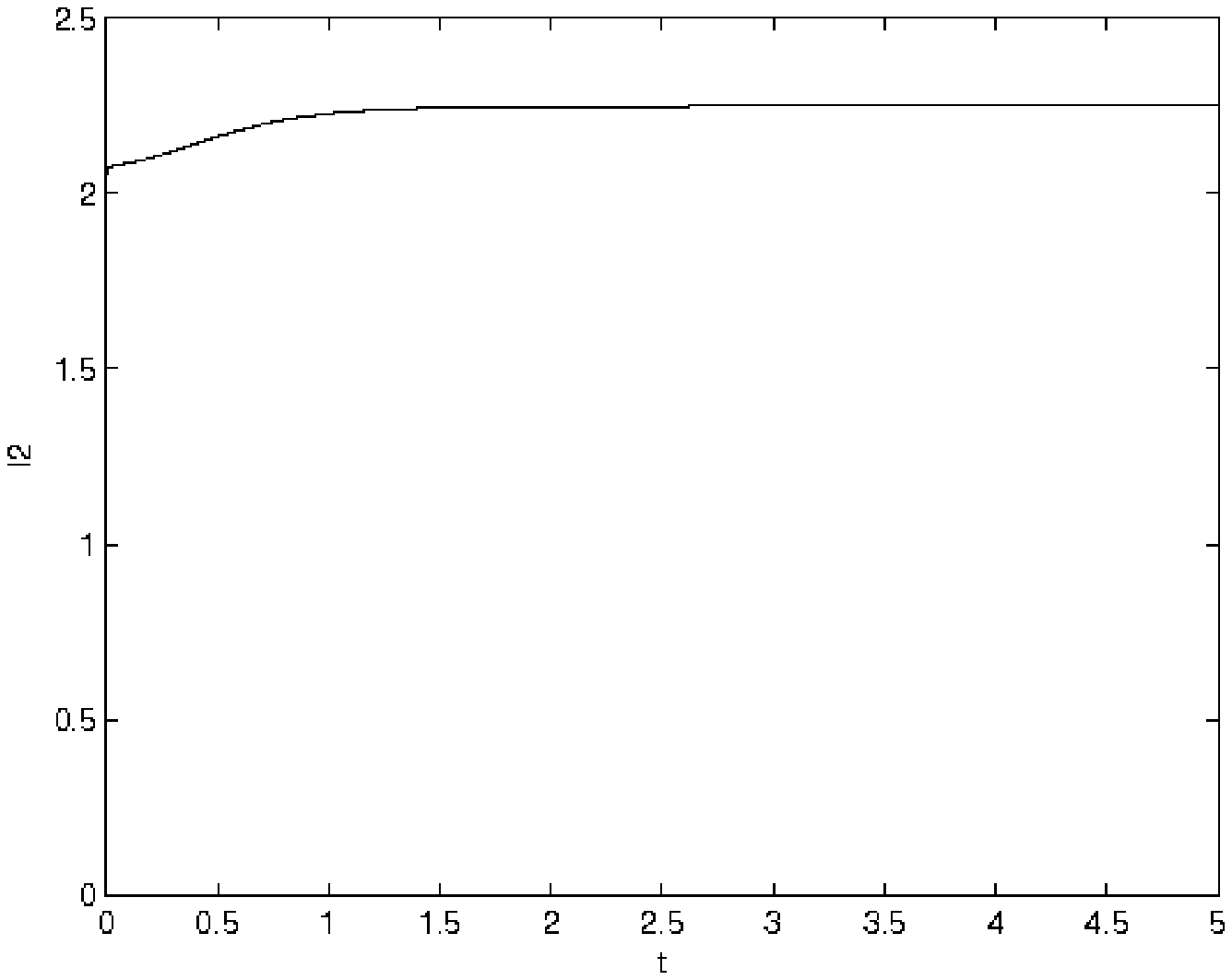}
\caption{ Dynamics of the total populations: $I_\ep^1(t)$(left) and $I_\ep^2(t)$(right), using the parameters in \ref{data1}. In both patches, the total population converges to a constant close to $2.25$.}
\label{fig:per}
\end{figure}
\begin{figure}[h]\begin{center}
\includegraphics[angle=0,scale=0.4]{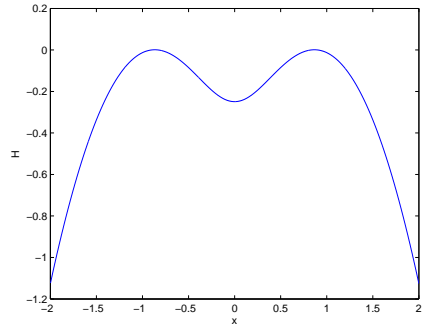}
\end{center}
\caption{The value of $H(\cdot,I_\ep^1(t),I_\ep^2(t)$ defined in \ref{def:H}, at time $t=5$.  }
\label{fig:H}
\end{figure}

\bigskip

\section*{Acknowledgments} The author benefits from a 2 year "Fondation Math\'ematique Jacques Hadamard" (FMJH) postdoc scholarship. She would like to thank Ecole Polytechnique for its hospitality. She is also grateful to Cl\'ement Fabre for useful discussions.


\end{document}